\documentclass[11pt]{article}
\usepackage{amsmath, amssymb, amsthm}
\usepackage{thmtools}
\usepackage{thm-restate}

\usepackage{mathtools}

\usepackage{color}

\usepackage[colorlinks=true,citecolor=black,linkcolor=black,urlcolor=blue]{hyperref}

\theoremstyle{plain}
\newtheorem{THM}{Theorem}[section]
\newtheorem*{THM*}{Theorem}
\newtheorem{PROP}[THM]{Proposition}
\newtheorem{LEMMA}[THM]{Lemma}

\newtheorem{COR}[THM]{Corollary}
\newtheorem{CLAIM}{Claim}

\theoremstyle{definition}

\theoremstyle{remark}

\newcommand{\pr}[1]{\mathbb{P} \left[ #1 \right]}
\newcommand{\prp}[1]{\mathbb{P}_{\mathcal{P} } \left[ #1 \right]}

\newcommand{\eps}{\varepsilon}

\newcommand{\gnd}{G_{n, d}}

\newcommand{\gndo}{G_{n, d_1}}
\newcommand{\gndt}{G_{n, d_2}}
\newcommand{\pnd}{\mathcal{P}_{n, d}}
\newcommand{\mofp}{\mathcal{M}(P)}
\newcommand{\pndt}{\mathcal{P}_{n, d_2}}

\newcommand{\piper}{\pi_{\alpha}}
\newcommand{\pf}{\mathcal{P}(F)}

\title{Anagram-free colorings of graphs}
	\date{}
 \author{
Nina Kam\v{c}ev \thanks{Department of Mathematics, ETH, 8092 Zurich. Email: nina.kamcev@math.ethz.ch.}
\and
Tomasz \L uczak \thanks{Faculty of Mathematics and Computer Science, Adam Mickiewicz University, Pozna\'n, Poland.
Email: tomasz@amu.edu.pl.
Research supported by NCN grant 2012/06/A/ST1/00261.
 }
\and
Benny Sudakov \thanks{Department of Mathematics, ETH, 8092 Zurich.
Email: benjamin.sudakov@math.ethz.ch.
Research supported in part by SNSF grant 200021-149111.}
}

\begin{document}
    \maketitle
    \setcounter{tocdepth}{3}
    
    \abstract{A sequence $S$  is called anagram-free if it contains no consecutive symbols $r_1 r_2\dots r_k r_{k+1} \dots r_{2k}$ such that  $r_{k+1} \dots r_{2k}$ is a permutation of the block $r_1 r_2\dots r_k$. Answering a question of Erd\H os and Brown, Ker\"anen constructed an infinite anagram-free sequence on four symbols.
        Motivated by the work of Alon, Grytczuk, Ha\l uszczak and Riordan~\cite{aghr}, we consider a natural generalisation of anagram-free  sequences for graph colorings. A coloring of the vertices of a given graph $G$ is called \emph{anagram-free} if the sequence of colors on any path in $G$ is anagram-free. We call the minimal number of colors needed for such a
coloring the \emph{anagram-chromatic} number of $G$.

		In this paper we study the anagram-chromatic number of several classes of graphs like trees, minor-free graphs and bounded-degree graphs. Surprisingly, we show that there are bounded-degree graphs (such as random regular graphs) in which anagrams cannot be avoided unless we basically give each vertex a separate color.
    
    \section{Introduction}
		The study of non-repetitive colorings was conceived by a famous result of Thue \cite{thue} from 1906. He showed that there exists an infinite sequence $S$ on an alphabet of three symbols in which no two adjacent blocks (of any length) are the same. In other words, $S$ contains no sequence of \emph{consecutive} symbols $r_1 r_2 \dots r_{2n}$ with $r_{i} = r_{i+n}$ for all $i \leq n$. Note that it is not a priori obvious that the minimal size of the alphabet necessary for an infinite non-repetitive sequence is even finite. Thue's result is interesting in its own right, but it also has influential and surprising applications, the most famous one probably occurring in a solution to the Burnside problem for groups by Novikov and Adjan \cite{novikov}. Thue-type problems lead to the development of Combinatorics on Words, a new area of research with many interesting connections and applications.
		 
		 Generalisations of Thue's result occurred in two directions. Firstly, the setting has been changed from sequences to, e.g., the real line, the lattice $\mathbb{Z}^n$, or graphs. Secondly, repetitions as a forbidden structure can be replaced by anagrams, sums, patterns etc. For a formal treatment and references to these problems, we refer the reader to the survey of Grytczuk \cite{gr08}. Here we focus  on graph colorings, and the structure we are avoiding are anagrams.
		 
		 A sequence $r_1 r_2\dots r_n r_{n+1} \dots r_{2n}$ is called an \emph{anagram} if the second block, $r_{n+1} \dots r_{2n}$, is a permutation of $r_1 r_2\dots r_n$. A long standing open question of Erd\H os \cite{erdosq} and Brown \cite{brown} was whether there exists a sequence on \{0, 1, 2, 3\} containing no anagrams. We call such sequences \emph{anagram-free}. It is easy to check that no such sequence on three symbols exists. In 1968 Evdokimov \cite{evdokimov} showed that the goal can be achieved with 25 symbols, which was the first finite upper bound. 
We remark that a finite bound can also be deduced from the Lov\'asz Local Lemma 
which, of course, has not been known in the time Evdokimov proved his result.
Later Pleasants \cite{pleasants} and Dekking \cite{dekking} lowered this number to five. Finally, Ker\"anen~\cite{keraenen} constructed arbitrarily long anagram-free sequences on four symbols using Thue's idea -- given a  non-repetitive finite sequence $S$ on symbols $\{0, 1, 2, 3\}$, we can replace each symbol by a longer word on the same alphabet in a way that yields a new, longer non-repetitive sequence $\overline{S}$. This answered the question of Erd\H os and Brown, but at the same time 
opened new avenues for further studies; some of them can be found in \cite{gr08}.
		 
		 Bean, Ehrenfeucht and McNulty \cite{bean} have studied the problem of non-repetitive colorings in a continuous setting. A coloring of the real line is called \emph{square-free} if no two adjacent intervals of the same length are colored in the same way. More precisely, for any  intervals $I = [a, b]$ and $J= [b,c]$ of the same length $L>0$, there exists a point $x \in I$ whose color is different from $x+L$. In \cite{bean}, they showed that there exist square-free two-colorings of the real line. Grytczuk \cite{gr08} also describes a strong variant of square-free colorings, which basically defines anagrams in the continuous setting and asks for an anagram-free coloring.
		 
		 Alon, Grytczuk, Ha\l uszczak and Riordan
		~\cite{aghr} proposed another variation on the non-repetitive theme. Let $G$ be a graph. A vertex coloring $c: V(G) \to \mathcal{C}$ is called \emph{non-repetitive} if any path in $G$ induces a non-repetitive sequence. Define the \emph{Thue number} $\pi(G)$ as the minimal number of colors in a non-repetitive coloring of $G$. 
It is easy to see that this number is a strengthening of the classical chromatic number, as well as the star-chromatic number. It turns out that the Thue number is bounded for several interesting classes of graphs, e.g. $\pi(P_n) \leq 3$ for a path $P_n$ of length $n$ (directly from Thue's Theorem), and $\pi(T) \leq 4$ for any tree $T$. 
Alon {\it et al.}~\cite{aghr} showed that $\pi(G) \leq c \Delta(G)^2$, where $c$ is a constant and $\Delta(G)$ denotes the maximum degree of $G$. They also found a graph $G$ with $\pi(G) \geq \frac{c^\prime \Delta^2}{\log \Delta} $. Closing the above gap remains an intriguing open question. 
Another interesting problem is to decide if the Thue number of planar graphs is finite. 
		 A survey of Grytczuk \cite{gr07} lays out some progress in this direction, as well as numerous related questions on non-repetitive graph colorings. 
		
		In the concluding remarks of their paper, Alon {\it et al.}~\cite{aghr}~ suggested investigating anagram-free colorings of graphs, which we do here.
		Let $G$ be a graph and let $c: V(G) \to \mathcal{C}$ be its vertex coloring. Two vertex sets $V_1$ and $V_2$ have \emph{the same coloring} if they have the same number of occurrences of each color, i.e.~$|c^{-1}(a) \cap V_1| = |c^{-1}(a) \cap V_2|$ for each $a \in \mathcal{C}$. An \emph{anagram} is a path $v_1v_2 \dots v_{2n}$ in $G$ whose two segments $v_1\dots v_{n}$ and $v_{n+1} \dots v_{2n}$ have the same coloring. We denote the minimum number of colors in an anagram-free coloring of $G$ by $\piper(G)$, and call it the \emph{anagram-chromatic number of $G$}. 
Clearly $\piper(G) \leq n$ for any $n$-vertex graph $G$. The result of Ker\"anen \cite{keraenen} states that $\piper(P_n) \leq 4$ for a path $P_n$ of length $n$, so it is only natural to ask what is $\piper$ for other families of graphs. It turns out that as soon as we move on from paths, the situation gets very different. We first show that the anagram-chromatic number of a binary tree already increases with the number of vertices.
		 \begin{PROP} \label{prop:trees} 
		    Let $T_h$ be a perfect binary tree of depth $h$, i.e.~every non-leaf has two children and there are $2^h$ leaves, all at distance $h$ from the root. Then
		    $$ \sqrt{\frac{h}{\log_2 h}} \leq \piper(T_h) \leq h+1.$$
		 \end{PROP}
		 It follows that the anagram-chromatic number of planar graphs is also unbounded, but it is still interesting to determine how quickly it increases with the number of vertices. 
We observe that in dealing with a family of graphs which 
admits small seperators (such as $H$-minor-free graphs), this fact can be used to
bound  $\piper(G)$ from above.

		 \begin{restatable}{PROP}{planar}		\label{prop:planar}				
			 Let $h \geq 1$ be an integer, and let $H$ be a graph on $h$ vertices. Any $n$-vertex graph $G$ with no $H$-minor satisfies $\piper(G) \leq 10h^{3/2} n ^{1/2}$. 
    	\end{restatable}
	
	In this paper we are particularly interested in  anagram-free colorings of  graphs of bounded degree. We show that, surprisingly,  there are  graphs of bounded degree
	such that to avoid anagrams  we essentially need to give every vertex a separate color. 
We show this by considering the \emph{random regular graph} $\gnd$, which is chosen uniformly at random from all $n$-vertex $d$-regular graphs. Here we write $\gnd$ for the sampled graph as well as the underlying probability space, and we study $\gnd$ for a constant $d$ and $n \rightarrow \infty$. We say that an event in this space holds \emph{with high probability} (whp) if its probability tends to $1$ as $n$ tends to infinity over the values of $n$ for which $nd$ is even (so that $\gnd$ is non-empty). Then our main result can be stated as follows.

	    \begin{THM} \label{thm:almostn}
            There exists a constant $C$ such that for sufficiently large $d$, with high probability, the random regular graph $G_{n,d}$ satisfies $$\left(1- \frac{C \log d}{d} \right)n \leq \pi_{per}(G_{n,d}) \leq \left(1- \frac{ \log d}{d} \right)n.$$
        \end{THM}
        
The rest of this paper is organized as follows. We start with some observations on the 
anagram-chromatic number for trees and minor-free graphs.
Then,  we give the proof of Theorem~\ref{thm:almostn}. We conclude the article with 
some open questions and conjectures on anagram-free colorings.   
	    
	    We mostly omit floor and ceiling signs for the sake of clarity. The $\log$ will denote the \mbox{base-$e$} logarithm. We will sometimes use standard $O$-notation for the asymptotic behaviour of the relative order of magnitude of two sequences, depending on a parameter $n \to \infty$.
			
		\section{Specific families of graphs}
        \subsection{Bounds for trees}
            A \emph{binary tree} is a tree in which every vertex has at most two children.
					 Let $T_h$ be a \emph{perfect} binary tree of depth $h$, that is to say that every non-leaf has two children and there are $2^h$ leaves, all at distance $h$ from the root. The root is taken to be at depth $0$, so a tree consisting of one vertex has depth $0$. Coloring each vertex of $T_h$ by its distance from the root shows that $\piper(T_h)\leq h+1$. In the following section, we will argue that actually any $n$-vertex tree can be colored with $2\log n$ colors. Proposition~\ref{prop:trees} asserts the lower bound
            $\piper(T_h) \geq \sqrt{\frac{h}{\log_2 h}}$,
            which will be proven in this section.

					Let $T$ be a vertex-colored binary tree and $U$ its subtree. The \emph{effective} vertices of $U$ are its root (i.e.~the vertex of $U$ of the smallest depth), leaves, and vertices of degree three. The \emph{effective depth of $U$} is set to $h_1$, where $h_1+1$ is the minimum number of effective vertices on any path from the root to a leaf (that is, the depth of the binary tree obtained by contracting all the internal degree-two vertices of~$U$).
					Note that if $U$ has effective depth $h_1$, then it has at least $2^{h_1}$ leaves. We say that $U$ is $\emph{essentially monochromatic}$ if all its effective vertices carry the same color.
					
					We will use a Ramsey-type argument to find a large essentially monochromatic subtree of a given tree. In the statement below $H(a_1, a_2, \dots a_d)$ 
					denotes the minimal number $h$ for which any perfect binary tree $T$ of depth $h$ whose vertices are colored using colors $1,2, \dots,d$ contains an essentially $i$-colored subtree of effective depth $a_i$, for some $i \in [d]$.
					
		\begin{LEMMA} \label{lemma:ramsey}
			$H(a_1, a_2, \dots a_d) \leq a_1 + \dots + a_d $.
		\end{LEMMA}
		
		\begin{proof}
			We use induction on $\sum_{i=1}^d a_i$. The base case is $a_1 = \dots = a_d = 0$, for which the claim clearly holds.
			
			Let $T$ be a perfect binary tree of depth $ a_1 + \dots + a_d $. Suppose that its root $v$ has the color 1, and call its children $v_L$ and $v_R$. Consider the subtrees $T_L$ and $T_R$ of depth at least $a_1 + \dots + a_d - 1$ rooted at $v_L$ and $v_R$ respectively. 
If for some $i \geq 2$, $T_L$ contains an essentially $i$-colored subtree of effective depth $a_i$, we are done. The same holds for $T_R$. Otherwise, using the induction hypothesis, $T_L$ and $T_R$ contain  essentially $1$-colored subtrees of effective depth $a_1 -1$. Those two subtrees, together with the root $v$, form an essentially $1$-colored subtree of $T$, as required.
		\end{proof}

		\begin{proof}[Proof of Proposition~\ref{prop:trees}]
			Let $T_h$ be colored using $d < \sqrt{\frac{h}{\log_2 h}}$ colors. By Lemma~\ref{lemma:ramsey}, it contains an essentially monochromatic subtree $U$ of depth ${h}/{d}$.

			Let $u$ be the root of $U$, and suppose $U$ is essentially red. There are 
at least $2^{h/d}$ paths from $u$ to the leaves, and the coloring of each path is a multiset of order at most $h+1$. On the other hand, there are at most $h^d$ such multisets. Since $h ^d < 2^{h/d}$ for our choice of $d$, there is a multiset which occurs on two different paths, say $P_1$ and $P_2$. Let $v$ be the lowest common vertex of $P_1$ and $P_2$, and let $\ell_1$ and $\ell_2$ be their respective leaves. By construction of $U$, the vertices $v$, $\ell_1$ and $\ell_2$ are red. Hence the segments from $\ell_1$ to $v$, excluding $v$, and from $v$ to $\ell_2$, excluding $\ell_2$, have the same coloring.
    
            We conclude that the given coloring of $T_h$, even restricted to $U$, contains an anagram.
		\end{proof}


                

			\subsection{Graphs with an excluded minor}
				Planar graphs are of special interest when it comes to coloring problems. The Four color theorem is one of the most celebrated results in Graph Theory. Moreover, the question of whether the Thue-chromatic number of planar graphs is finite has attracted a lot of attention and is still open. 
We use separator sets to show that for a large class of minor-free graphs the 
anagram-chromatic number is of order $O \left( \sqrt{n} \right)$. The crucial ingredient 
of our argument is  
the separator theorem, proved by Alon, Seymour and Thomas \cite{seymour}. 
It states that for a given $h$-vertex graph $H$, in any graph $G$ with $n$ vertices and no $H$-minor, one can find a set $S \subset V(G)$ order $|S| \leq h^{\frac 32}n ^{\frac 12}$, whose removal partitions $G$  into disjoint subgraphs each of which has at most $\frac{2n}{3}$ vertices. Such a set $S$ is called a \emph{separator} in $G$. 
				
					Using this theorem, we construct a coloring of any proper minor-closed family of graphs. For convenience of the reader, we restate Proposition~\ref{prop:planar}.
					
					\planar*
					\begin{proof}
						The coloring is inductive -- suppose the claim holds for graphs on at most $n-1$ vertices.
						Let $G$ be as in the statement, and let $S$ be a separating set of vertices in $G$ of order at most $h^{3/2}n ^{1/2}$ given by the Separator Theorem. 
Then $G-S$ consists of two vertex-disjoint subgraphs spanned by $A_1 \subset V(G)$ and $A_2 \subset V(G)$, with $|A_i |\leq \frac{2n}{3}$.
						
						The induced subgraphs $G[A_i]$ do not contain $H$ as a minor, so by the inductive hypothesis, we can color them using $k=10h^{3/2}\sqrt{{2n}/{3}}$ colors $a_1, a_2, \dots, a_k$. Note that the two subgraphs receive colors from the same set. This coloring guarantees that any path containing only vertices from $A_1$ or $A_2$ is anagram-free. Furthermore, we assign to each vertex $v_i \in S$ a separate color $b_i$, making any path passing through $S$ anagram-free.  Hence the coloring is indeed anagram free. As intended, the number of colors used is at most
						$$h^{3/2}n^{1/2} \left(10\sqrt{\frac{2}{3}}  +1 \right)\leq 10h^{3/2}n^{1/2}.$$
					\end{proof}
					
				Since 	planar graphs are characterized as graphs containing neither $K_5$ nor $K_{3, 3}$ as a minor, we arrive at the following consequence of the above result (note that the constant 150
				can be replaced by 19 if we use the fact that each planar graph has a separator of 
				order $1.84\sqrt{n}$). 
				\begin{COR}
					Let $G$ be an $n$-vertex planar graph. Then $\piper(G) \leq 150\sqrt{n}$.
				\end{COR}
                
                    In fact, any hereditary family of graphs with small separators can be colored using the argument from Proposition~\ref{prop:planar}. For example, it is easy to see  
that an $n$-vertex forest $F$ 
contains a single vertex which separates it into several forests on at most $n/2$ vertices. The same inductive argument implies $\piper(F) \leq \lceil \log_2 n\rceil$. 

As for the lower bound for planar graphs, we only have the following modification of the argument we gave for trees.
                
                \begin{PROP} \label{prop:planar2}
        There is an $n$-vertex planar graph $F_n$  
with $\piper(F_n) \geq \lceil \frac14\log_2 n \rceil$.
                \end{PROP}
                
                \begin{proof}
                    Let $F_n$ be a perfect binary tree with $n$ leaves, plus extra edges between any two vertices on the same level having the same parent. Suppose it is colored in $k= \lceil \frac14\log_2 n \rceil$ colors. 
                    The number of shortest paths from the root to the vertices 
corresponding to leaves is $n$, whereas the number of possible colorings of these paths is $\binom{\log_2 n+ k}{k-1}<n$;
hence some two paths have the same coloring. These two paths, minus the shared initial segment,
 can be made into an anagram.
                \end{proof}

        \section{Bounded-degree graphs}
            \subsection{A four-regular graph with a large anagram-chromatic number}

In this section we study the number of colors needed to color a bounded-degree graph on $n$ vertices so as to avoid all anagrams. The trivial
upper bound is $n$, so we will mainly be interested in lower bounds for the
anagram-chromatic number. It is easy to use the Local Lemma to show that for every 
graph $G$ with maximum degree at most two, we have $\piper(G)\le C$ for a suitably
chosen constant. It turns out that there are already  $4$-regular graphs $G$ for which
$\piper(G)$ grows rather quickly with the size of the graph. 
				 				 
	\begin{PROP}
                For infinitely many values of $n$, there exists a 4-regular $n$-vertex graph $H$ with $\piper (H) \geq  \frac{\sqrt{n}}{ \log_2 n}  $. 
            \end{PROP}
            
            \begin{proof}
                  Note that for each even $k\geq 4$, there exists a 3-regular $k$-vertex graph $G$ which is Hamilton-connected, which means that any two vertices of $G$ are joined by a Hamilton path. Indeed, it can be easily checked that for any $m \geq 1$, the Cayley graph of $C_2 \times C_{2m+1}$ with canonical generators  is Hamilton-connected. For a self-contained proof, we refer the reader to \cite{cq}. Let $n=(k+1)k$. Take $k+1$ copies of  such $G$ on vertex sets $V_1, V_2, \dots V_{k+1}$ with $|V_i|=k$. Furthermore, take a perfect matching $M$ on $ V_1 \cup  \dots \cup V_{k+1}$ such that there exists exactly one edge between any two $V_i$ and $V_j$, for $i \neq j$. To see that such a matching exists, denote $V_i = \{v_{ij}: j \in [k+1]\setminus \{i \}\}$, and take $M = \{ \{v_{ij}, v_{ji} \}: 1\leq i <j \leq k+1 \}$.
                  
                  Call the resulting graph $H$. $H$ is 4-regular - any vertex has three adjacent edges belonging to its copy of $G$ and one edge belonging to $M$. 
                  Suppose that the vertices of $H$ are colored with $ \left \lfloor \frac{\sqrt{n}}{ \log_2 n}  \right \rfloor $ colors.  Consider the subsets of form $\bigcup_{i \in S}V_i$ for any $S \subset [k+1]$. There are $2^{k+1}$ such subsets. The coloring of each $\bigcup_{i \in S}V_i$ defines a multiset of order at most $n$. Given $ \left \lfloor \frac{\sqrt{n}}{ \log_2 n}  \right \rfloor $ colors, the number of such multisets is at most $n^{ \frac{\sqrt{n}}{ \log_2 n}  } = 2^{\sqrt{n}}< 2^{k+1}$.
    Thus, by the pigeonhole principle, there are two distinct sets $S, T\subset [k+1]$ such that $\bigcup_{i \in S}V_i$ and $\bigcup_{i \in T}V_i$ have the same number of occurrences of each  color.  
    The same holds for sets $S^\prime = S \setminus T$ and $T^\prime = T \setminus S$, which are in addition disjoint. Without loss of generality assume $S^\prime = \{V_1, \dots V_s \}$ and $T^\prime = \{V_{s+1}, \dots V_{2s} \}$. By the choice of $M$, we can find vertices $v_1, u_1, v_2, u_2, \dots v_{2s} , u_{2s}$ such that $v_i, u_i \in V_i$ for $i \in [2s]$,and $u_i v_{i+1}$ are edges in $M$ for $i \in [2s-1]$. Moreover, we can find a Hamilton path in each $H[V_i]$ between $u_i$ and $v_i$, using Hamilton-connectedness of $G$. Concatenating these $2s$ paths gives us a path in $H$ which traverses $V_1 \cup V_2 \dots \cup V_{2s}$ in order. This path forms an anagram in $H$, so $\piper (H) > \left \lfloor \frac{\sqrt{n}}{ \log_2 n}  \right \rfloor$.
            \end{proof}

        \subsection{Random regular graphs}
            Let us start with a simple observation which slightly improves  the trivial 
            upper bound $n$ for the anagram-chromatic number of a graph. 

            \begin{PROP} Let $G$ be an $n$-vertex graph with an independent set of order $m$.
            Then $\piper(G) \leq n-m+1$. \label{prop:indep}
            \end{PROP}
            \begin{proof}
                Let $S$ be an independent set inside $G$ of order $m$. Give each vertex of $S$ the same color, and each vertex of $V(G) \setminus S$ its own color. Any path in $G$ contains at least one vertex of $V(G) \setminus S$, so it cannot contain an anagram. This means that our coloring is indeed anagram-free.
            \end{proof}
            
The above bound is essentially optimal for the random regular graph $G_{n,d}$. To recapitulate, Theorem~\ref{thm:almostn} states that for sufficiently large $d$, with high probability, $\gnd$ satisfies $$\left(1- \frac{2 \cdot 10^5\log d}{d} \right)n  \leq \piper(G_{n,d}) \leq \left(1- \frac{ \log d}{d} \right)n.$$

				The upper bound is an immediate consequence of Proposition~\ref{prop:indep}, and the 
				fact that with high probability, $\gnd$ contains an independent set of order asymptotic to $\frac{2 \log d}{d}n > \frac{ \log d}{d}n  $ (see, for instance, Frieze and {\L}uczak~\cite{frieze}). We will now outline the proof of the lower bound on $\piper(\gnd)$, which comprises the remainder of the section. 				
				Instead of studying  the random $d$-regular graph $\gnd$, we will consider  the union of 
two random graphs $G_{n, d_1}$ and $ G_{n, d_2}$ with  $d = d_1 + d_2$. The asymptotic 
properties of $\gnd$ are contiguous with such a model (see Lemma 9.24 in \cite{JLR}).
Let $G_1 = \gndo$,  $G_2 = \gndt$ and $c$ be a given vertex-coloring of $G = G_1 \cup G_2$. The first step is to find two vertex subsets $V_1$ and $V_2$ with the same coloring such that $G_1[V_1]$ and $G_1[V_2]$ have good expansion properties. Then we use the edges of $G_2$ to extend paths on $V_1$ and $V_2$, eventually building Hamilton cycles $C_1$ in $G [V_1]$ and $C_2$ in $G[V_2]$. Finally, we can find an edge $v_1 v_2 \in G$ with $v_i \in V_i$ and use it to build a single path $S$ which traverses first the vertices of $C_1$ and then the vertices of $C_2$. The segments $S[V_1]$ and $S[V_2]$ give an anagram in $c$.


				Before proceeding, let us introduce some notation.						
				For a graph $G$ and $v \in V(G)$, we denote the neighborhood of $v$ in $G$ by $N_G(v)$. For a vertex set $U \subset V(G)$, $N_G(U) =\bigcup_{v \in U} N_G(v) \setminus U$. The graph induced on $U$ is $G[U]$, and its edge set is denoted by $E_G(U) = E(G[U])$. For disjoint sets $U$ and $T$, $E_G(U, T)$ is the set of edges with one endpoint in $U$ and one in $T$. Finally, the corresponding counts are $e_G(U) = |E_G(U)|$ and $e_G(U, V) = |E_G(U, V)|$.
				We denote the uniform probability measure on the space of random regular graphs $G_{n, d}$ by $\mathbb{P}$, suppressing the indices. All the inequalities below are supposed to hold only for $n$ 
large enough.

						\subsubsection{Edge distribution in the configuration model}
							In analysing $\gnd$, we pass to the \emph{configuration} model of random regular graphs. For $nd$ even, we take a set of $nd$ points partitioned into $n$ cells $v_1,\, v_2,\, \dots v_n$, each cell containing $d$ points. A perfect matching  $P$ on $[nd]$ induces a multigraph $\mathcal{M}(P)$ in which the cells are regarded as vertices and pairs in $P$ as edges. 
							For a fixed degree $d$ and $P$ chosen uniformly from the set of perfect matchings $\pnd$, the probability that $\mathcal{M}(P)$ is a simple graph is bounded away from zero, and each simple graph occurs with equal probability. Therefore, if an event holds whp in $\mathcal{M}(P)$, then it holds whp even when we condition on the event that $\mathcal{M}(P)$ is a simple graph, and therefore it holds whp in $\gnd$ (for a formal description of the configuration model
and its basic properties, see, for instance, Chapter 9 of~\cite{JLR}). 
				
				We use the configuration model to get a bound on the edge distribution in $\gnd$ analogous to the Erd\H{o}s-Renyi model. The uniform probability measure on $\pnd$ is denoted by $\mathbb{P}_\mathcal{P}$. Both indices $n$ and $d$ are kept so that each perfect matching $P$ corresponds to a unique $d$-regular multigraph $\mofp$.
				
				\begin{LEMMA} \label{lemma:edges}
				    Let $V_1 \subset [n]$, and let $B$ be a set of pairs of vertices from  $V_1$. 
Let $E$ be another set of pairs of vertices from $[n]$ with $|E| \leq \min \left\{ \frac {1} {4 |V_1|} |B| d,\, \frac{nd}{20} \right \}$. For a fixed positive integer $d$ and $P \in \pnd$ chosen uniformly at random,
				    $$\prp{\mofp \supset E \text{ and } \mofp \cap B = \emptyset} \leq
				    \left(\frac{2d}{n} \right)^{|E|} e^{-\frac{2|B|d}{5n}}.$$
				\end{LEMMA}
						The lemma also holds for more general configurations of $B$ and $E$, but we state it in the form which is fit for our purpose. We will need a bound on the probability that $\gnd$ does not intersect a given set of edges. For this purpose, we use the following lemma. 
					
				\begin{LEMMA} \label{lemma:avoid}
				    For each even number $N$, let $F = F(N)$ be a graph on $[N]$ consisting of at least $\beta N^2$ edges. Let $G_{N, 1}$ denote a random matching on $[N]$. Then 
				    $$\pr{G_{N, 1} \cap F = \emptyset} \le e^{-\frac{8\beta}{9} N}.$$
				\end{LEMMA}
					
				\begin{proof} 
				    Let $\pf$ be the set of perfect matchings on $[N]$ which do not intersect our graph $F$. Since the number of perfect matchings on $[N]$ is exactly 
$\frac{N!}{2^{\frac N2} \left(\frac N2 \right)!} 
$, we need to show that
				    $$|\pf| \leq e^{ -\frac{8\beta}{9} N} \frac{N!}{2^{\frac N2} \left(\frac N2 \right)!}.$$
						Consider the complement $\overline{F}$ of $F$. The matchings in $\pf$ are exactly perfect matchings in $\overline{F}$. We use the following estimate of Alon and Friedland~\cite{friedland}, which is a simple corollary of the Br\`egman bound on the permanent of a $(0,1)$-matrix.
							\begin{THM}[\cite{friedland}] 
							    Let $H$ be a graph on $[N]$. Let $r_1, r_2, \dots r_N$ be the degrees of the vertices in~$H$. Furthermore, denote $r = \frac{1}{N} \sum_{i=1}^N{r_i}$.
							    Then the number of perfect matchings in $H$ is at most
							    $$\prod_{i=1}^N(r_i !) ^{\frac{1}{2r_i}} \leq (r!)^{\frac{N}{2r}}.\quad \qed$$
							\end{THM}
						    
						    
						  We apply this bound directly to the graph $\overline{F}$ with $r = N - 2\beta N$, and use Stirling's formula to reach the final result. Indeed, we have
						  
						    \begin{align*}
						      | \pf|  &\leq \left( (N - 2\beta N)! \right)^{\frac{1}{2 \left(1-2\beta \right)}} \\ 
						       \pr{G_{N, 1} \cap F = \phi} & \leq \left( (N - 2\beta N)! \right)^{\frac{1}{2(1-2\beta)}} \cdot \frac{ \left(\frac N2 \right)! \cdot 2^{\frac N2}}{N!}\\
						            & = O\left(\sqrt{N}\right) \left( \frac{(1-2\beta)N}{e} \right)^{\frac N2}  \left(\frac eN \right)^{\frac N2} 
						            = O\left(\sqrt{N}\right) e^{-\beta N }.\\
						    \end{align*}
						        
				    Here we use the fact that $\frac{ N!}{
				    \left(\frac N2 \right)! \cdot 2^{\frac N2}} = \Theta(1) \left(\frac Ne \right)^{\frac N2}$, as well as the inequality $e^{1-2\beta} \leq e^{-2\beta}$. Absorbing the error term into the constant, we get, for $N$ large enough,
		            
		            $$\pr{G_{N, 1} \cap F = \emptyset} \leq e^{-\frac{8\beta}{9} N}.$$
				\end{proof}
					
				\begin{proof}[Proof of Lemma~\ref{lemma:edges}]
						We will restate the event $ \{ \mofp \supset E \text{ and } \mofp \cap B = \emptyset \}$ in terms of $P$, rather than $\mathcal{M}(P)$. For a matching $M \subset \binom{[nd]}{2}$, we denote the induced multigraph on $V = \{v_i \}_{i \in [n]}$ by $\mathcal{M}(M)$. To save on notation, we write $\mathcal{M}(M)$ for both the graph and its edge set. Conversely, if $e = \{v_i, v_j\}$ is a pair of vertices from $V$, we denote its corresponding pairs
in $\binom{[nd]}2$ by $\tilde{e} = \{\{x, y\}: x \in v_i, y \in v_j\}$. Finally, for a set $E \subset \binom{V}{2}$, we put $\tilde{E} = \bigcup_{e \in E} \tilde{e}$.
						
						Assume that $\mathcal{M}(P) \supset E$. Then we can find a matching $M \subset P$ such that $|M|= |E| = m$ and $\mathcal{M}(M)=E$. Conditioning over the possible choices of  $M$, we have
						\begin{align*}	
							&\prp{\mathcal{M}(P) \supset E  \wedge \mathcal{M}(P) \cap B   = \emptyset } \leq  \sum_{M} \prp{P \supset M}\prp{P \cap \tilde{B}  = \emptyset \mid P \supset M}.
						\end{align*}
						
						We bound the two probabilities separately. Fix a choice $M = \{ \{x_i, y_i\} : i \in [m]\}$, and let $W = [nd] \setminus\{x_1, \dots, x_m, y_1, \dots, y_m\}. $
						
						\begin{CLAIM}	\label{cl1}
							$\prp{P \supset M} \leq \frac{2}{(nd- 2m)^m}$
						\end{CLAIM}
						
						To show this, we just count perfect matchings. The total number of perfect matchings $P$ is $\frac{(nd)!}{\left(\frac{nd}{2}\right)!2^{\frac{nd}{2}}}$. The points from $W$ can be paired in 
            $\frac{(nd-2m)!}{\left(\frac{nd}{2}-m\right)!2^{\frac{nd}{2}-m}}$ ways. Altogether, using Stirling's formula, we get

                \begin{align*}
                    &\prp{M  \subset P} \leq  \frac{\left(nd -2m\right)!\left(\frac{nd}{2}\right)!}{\left(nd \right)!\left(\frac{nd}{2}-m\right)!2^{-m}} \\
                                    =&(1+o(1)) \left(\frac{nd-2m}{nd} \right)^{nd} \left(\frac{nd-2m}{e} \right)^{-2m} \left(\frac{nd}{nd-2m} \right)^{\frac{nd}{2}} \left( \frac{nd-2m}{e} \right)^m \\
                                    =&(1+o(1)) \left(1 - \frac{2m}{nd}\right)^{\frac{nd}{2}} \left( \frac{e}{nd-2m} \right)^m
                                    \leq \frac{2}{(nd- 2m)^m}.
                \end{align*}
             Here we used the fact that since $1-x \leq e^{-x}$, we have $(1-\frac{2m}{nd})^{\frac{nd}{2}} \leq e^{-m}$.
						
						\begin{CLAIM}\label{cl2}
							For $|B|= \beta n^2$, $\prp{P \cap \tilde{B}  = \emptyset \mid P \supset M} \leq e^{-\frac{2\beta}{5}nd}.$
						\end{CLAIM}
							Let $W$ be as before, and denote $N = |W|$. Using the assumption $m \leq \frac{nd}{20}$, we get $N = nd - 2m \geq \frac{9nd}{10}$ . Let $B_W = \tilde{B}[W]$, that is, the set of pairs contained in $W$ which would induce~$B$. By putting the matching $M$ aside, we have lost some pairs from~$\tilde{B}$, namely those touching the vertices of $M$. Each vertex of $M$ is contained in at most $|V_1|d$ pairs from~$\tilde{B}$, so the hypothesis $m \leq \frac{1}{4|V_1|}|B|d$ implies
							$$
								|B_W| 
								\geq |B|d^2 - 2m |V_1| d = \beta n^2 d^2 \left(1 - \frac{2m |V_1|}{\beta n^2 d} \right) \geq \frac{1}{2}\beta n^2 d^2 \geq \frac 12 \beta N^2.
							$$
						
						A random matching $P$ conditioned on $P \supset M$ corresponds to a random matching on $W$, i.e.~an element of $G_{N, 1}$. Hence we can apply Lemma~\ref{lemma:avoid} with $|B_W| \geq \frac 12 \beta N^2 $, and $N = |W| \geq \frac{9nd}{10}$.
						$$\prp{P \cap B_W = \emptyset \mid P \supset M} \leq \pr{E(G_{N, 1}) \cap B_W = \emptyset} \leq e^{-\frac 89 \cdot \frac 12 \beta N}  \leq e^{-\frac{2}{5}\beta nd}.$$

						Claim~\ref{cl1} and \ref{cl2} hold for any choice of the matching $M$ with $\mathcal{M}(M) = E$. Putting them together, and using the fact that there are at most $d^{2m}$ such matchings $M$, we get
						$$\prp{\mathcal{M}(P) \supset E \text{ and } \mathcal{M}(P) \cap B = \emptyset } \leq d^{2m} \cdot \frac{2}{(nd- 2m)^m} \cdot  e^{-\frac{2\beta}{5}nd}. $$
						 Using $m \leq \frac{nd}{20}$,
						$$\prp{\mathcal{M}(P) \supset E \text{ and } \mathcal{M}(P) \cap B = \emptyset } \leq d^{2m} \cdot \left(\frac{2}{nd} \right)^m  e^{-\frac{2\beta}{5}nd} =
					    \left(\frac{2d}{n} \right)^m  e^{-\frac{2\beta}{5}nd}.$$
				\end{proof}

				\subsubsection{Expansion properties and P\'osa rotations}
				    Recall that we will be working with the union of random graphs $\gndo$ and $\gndt$. First we focus on expansion properties of $G_1 = \gndo$, which will allow us to do rotations in $G_1$. 
The aim is  to identify large sets of vertex pairs, called boosters,  
 which could increase the length of the longest path in $G_1$. Hence all the lemmas in this section will later be applied with $d$ replaced by $d_1$.
	The following lemma says that edges in $\gndo$ are uniformly distributed.
				    			    			    
				\begin{LEMMA} \label{lemma:quasirandom}
For sufficiently large $d$, with high probability $\gnd$ has 
the following two properties:
				    \begin{description} \itemsep-1pt
				        \item{(P1)} any vertex set $U$ with $|U|\leq \frac{30\log d}{d}n$ satisfies
				        $e_G(U) \leq 100 |U| \log d $;
				        \item{(P2)} for any two disjoint vertex subsets $T$ and $U$ with $|T|\geq \frac{10\log d}{ d}n $ and $|U|\geq \frac{100\log d}{ d}n $, we have
				    \begin{align}
    				    e_G(T, U) \geq 
    				   |T||U| \frac{d}{20n} . \label{twosets}
    				\end{align}
								
				    \end{description}
				\end{LEMMA}
				
				\begin{proof}
				    We prove Lemma~\ref{lemma:quasirandom} for $G$ sampled according to the configuration model, i.e.~take $G = \mathcal{M}(P)$, where $P$ is a random element of $\pnd$. We start with (P2). Take vertex sets $T$ and $U$ in~$G$ with $|T| = t$ and $|U| = u$. We need to bound the probability of the event
					$$D_{T, U} = \left \{ e_{G}(T, U) < \frac{d}{20n}tu \right \}.$$
					For a fixed set of $m$ edges $E$ with $m \leq \frac{d}{20n}tu $, the probability that $E_{G}(T, U) = E$ is at most
					$$\left(\frac{2d}{n} \right)^m e^{-\frac{2d}{5n}(tu-m)}.$$
					This bound is a direct application of Lemma~\ref{lemma:edges} to the edge set $E$ and its bipartite complement $(T \times U) \setminus E$.
					Taking the union bound over all sets $E$, we get
					
					$$\prp{ D_{T, U}} \leq \sum_{m=0}^{\frac{d}{20n}tu} \binom{tu}{m}\left(\frac{2d}{n} \right)^m e^{-\frac{2d}{5n}(tu-m)}
					\leq
					\sum_{m=0}^{\frac{d}{20n}tu} \left(\frac{etu}{m} \cdot \frac{2d}{n} \right)^m e^{-\frac{2d}{5n}(tu-m)}.$$
					The summand is increasing in $m$, so we bound it using the largest term, $m = M =  \frac{d}{20n}tu $.
			        
			        \begin{align*}
			            \prp{ D_{T, U}} &\leq M
					\left(\frac{etu\cdot 20n}{dtu} \cdot \frac{2d}{n} \right)^{\frac{d}{20n}tu} e^{-\frac{2d}{5n}(tu-M)}
					= M \cdot \left(40e \right)^{\frac{dtu}{20n} }e^{-\frac{2d}{5n}(tu-M)}\\
					&= Me^{\frac{dtu}{n}\left(\frac{5}{20}-\frac 25 + \frac{d}{n} \right) }
					\leq e^{-\frac{dtu}{8n}}.
			        \end{align*}
							
							Finally, we take the union bound over all sets $T$ and $U$ of order at least $t_0 = \frac{10\log d}{d}n$ and $u_0 = \frac{100\log d}{d}n$ respectively.
							$$\prp{G \text{ violates (P2)}} \leq \sum_{t = t_0}^n \sum_{u = u_0}^n \binom{n}{t} \binom{n}{u} e^{-\frac{dtu}{8n}}.$$
							We use the bound $\binom{n}{t} \leq d^t = e^{t\log d}$ valid for $t \geq t_0$ and large enough $d$.
							$$\prp{G \text{ violates (P2)}} \leq \sum_{t = t_0}^n \sum_{u = u_0}^n e^{t\log d + u \log d}  e^{-\frac{dtu}{8n}} .$$
							For $t  \geq \frac{10\log d}{d}n$, we get $u\log d  \leq \frac{dtu}{10n}$. Similarly, since $u  \geq \frac{100\log d}{d}n$, it holds that $t\log d  \leq \frac{dtu}{100n}$.
							$$\prp{G \text{ violates (P2)}} \leq O(n^2) e^{\left(\frac{1}{100}+ \frac{1}{10} - \frac{1}{8}\right)\frac{n\log^2 d }{d}}
													=o(1).$$
							
							We deduce (P1) from the following more  general statement.
					\begin{CLAIM} \label{claim:rc}
					    Fix the constants $A_1$ and $A_2$ satisfying $\left(\frac{e A_1}{A_2} \right)^{A_2} \leq e^{-2}$. Then with high probability, any vertex set $U \subset V(G)$ with $|U|\leq \frac{A_1\log d}{d}n$ satisfies
				        $e_G(U) \leq A_2 |U| \log d $.
					\end{CLAIM}
			            Introducing $A_1 = 30$ and $A_2 = 100$, which indeed satisfy $\left(\frac{30e}{100} \right)^{100} \leq e^{-2}$, gives exactly (P2).
			            
			            To prove the claim, we fix a set $U$ of order $u$, and use Lemma~\ref{lemma:edges} to establish that the probability of some $A_2 u \log d$ edges occurring in $U$ is at most
			            $$\binom{u^2 / 2}{A_2 u \log d} \left(\frac{2d}{n} \right)^{A_2 u \log d}
			            \leq \left( \frac{eu }{2 A_2 \log d } \cdot \frac{2d}{n} \right)^{A_2 u \log d}.$$
			            
			            Let $D_u$  denote the event that some subset $U$ with $|U| = u$ spans more than $A_2 u \log d$ edges. We have
			            \begin{equation}\label{eq1}
			                \prp{D_u} \leq \binom{n}{u} \left( \frac{e u d}{ A_2n \log d }  \right)^{A_2 u \log d} 
			                \leq \left[\frac{ne}{u} \left( \frac{eu d}{ A_2n \log d }  \right)^{A_2 \log d } \right]^u.
			            \end{equation}
			                The term in square brackets is increasing in $u$, so for $u \leq \frac{A_1\log d}{d}n$, 
			            \begin{align*}
			                \prp{D_u} \leq \left[\frac{ed}{A_1 \log d} \left( \frac{e A_1}{ A_2 }  \right)^{A_2 \log d }\right]^u 
			                \leq \left[\frac{ed}{A_1 \log d} e^{-2 \log d }\right]^u < d^{-u}.
			            \end{align*}
			            
			            Here we used the condition $\left(\frac{e A_1}{A_2} \right)^{A_2} \leq e^{-2}$. For $u \leq \sqrt{n}$ we use (\ref{eq1}) to get a stronger bound
			            $$\prp{D_u} \leq \left(O\left( n^{\frac{1}{2}(1-A_2 \log d)} \right) \right)^u < n^{-1},$$
valid for large $d$. Putting the two bounds together,
			           $$ \prp{G \text{ violates (P1)}} \leq \sum_{u = 1}^{\sqrt{n}} \prp{D_u}+ \sum_{u = \sqrt{n}} ^{\frac{A_1\log d}{d}n} \prp{D_u} \leq  \sqrt{n} \cdot n^{-1} + \sum_{u = \sqrt{n} }^{\frac{A_1\log d}{d}n} d^{-u} = o(1),$$
			           completing the proof of Claim~\ref{claim:rc}.			           
			           To recapitulate, applying the claim for $A_1 = 30$ and $A_2 = 100$ gives that $G = \mathcal{M}(P)$ satisfies (P1) with high probability.
			           
			           Since the random graph $\gnd$ is contiguous to $G$, we conclude that for large enough $d$, $\gnd$ satisfies (P1) and (P2).
			            
				\end{proof}						
				    The next step is to build subsets of $[n]$ which will later give us the required anagram. In everything that follows, take $\alpha =  10^5$. Given a $d$-regular graph $G$, we say that a subset $V_1 \subset [n]$ is \emph{$G$-dense} if $\frac{\alpha \log d}{ 2d}n \geq|V_1| \geq \frac{\alpha \log d}{4 d}n$, and any vertex $v \in V_1$ has at least $\frac{\alpha}{160}\log d$ neighbors in~$V_1$.

				\begin{LEMMA} \label{lemma:splitting}
				    Suppose we are given a $d$-regular graph $G$ on $[n]$ with properties (P1) and (P2), and a vertex coloring $c: [n] \to \mathcal{C}$ with $|\mathcal{C}| = \left(1- \frac{\alpha \log d}{d} \right)n$ colors. For sufficiently large $d$ and $n$, there exist two disjoint $G$-dense sets of vertices $V_1, V_2 \subset [n]$ which have the same coloring.\\
				    
				\end{LEMMA}
				
				\begin{proof} Let $c$ be a coloring of the vertices of $G$ into 
	$\left(1- \frac{\alpha \log d}{d} \right)n$ colors.
						\begin{CLAIM} There exists a subset $Z \subset V(G)$ satisfying $\frac{\alpha \log d}{d}n \geq |Z| \geq \frac{\alpha \log d}{2d}n $ such that each color appears in $Z$ an even number of times, and for all $v \in Z$, $|N_G(v) \cap Z | \geq \frac{\alpha}{40} \log d$.
						\end{CLAIM}
        We construct $Z$ using the following algorithm. Let $V(G) = [n]$, and denote $\delta =  \frac{\alpha}{40} \log d$. Let $X$ contain one vertex from each color class with an odd number of colors, so $|X| \leq \left(1- \frac{\alpha \log d}{d} \right)n$. We assume $|X| = \left(1- \frac{\alpha \log d}{d} \right)n$, by discarding more pairs of vertices of the same color if necessary. Furthermore, set $R_0 = \hat{R}_0 = \emptyset$. Note that from this step onwards, all color classes in $[n] \setminus (X \cup R_i \cup \hat{R}_i)$ will have even order. For $i \geq 0$, we form $R_{i+1} := R_i \cup \{v\}$, $\hat{R}_{i+1} = \hat{R}_i \cup \{w \}$, where $v$ is the smallest vertex with fewer than $\delta $ neighbors in $[n] \setminus (X \cup R_i \cup \hat{R}_i)$, and $w$ is the smallest vertex with $c(v)= c(w)$. When there are no such vertices $v$, set $Z = [n] \setminus (X \cup R_i \cup \hat{R}_i)$ and terminate the algorithm. We claim that this occurs after at most $ \frac{10 \log d}{d}n $ steps.

        Suppose that it is not the case and $G$ satisfies (P1) and (P2), but the algorithm continues beyond $t =  \frac{10 \log d}{d} n $ steps. Look at the sets $R_t$ and $Z_t =  [n] \setminus (X \cup R_t \cup \hat{R}_t)$. Each vertex from $R_t$ has fewer than $\delta$ neighbors in $Z_t$, so
        $$e_G(R_t, Z_t) < \delta |R_t| =  \frac{\alpha}{40} |R_t| \log d .$$
        
        On the other hand, since $|R_t| = t = \frac{10\log d}{d}n$ and $|Z_t|  =  \frac{\alpha\log d}{d}n - 2t \geq  \frac{\alpha\log d}{2d}n$, the property (P2) gives
        $$e_G(R_t, Z_t) \geq |R_t||Z_t|\frac {d}{20n} \geq |R_t| \cdot \frac{\alpha \log d}{2d} \cdot \frac{d}{20n} = \frac{\alpha}{40}  |R_t| \log d.$$
        We reached a contradiction, so indeed we have the desired set $Z$ with $|Z| \geq \frac{\alpha \log d}{2d}n $.
				
				We show the existence of the required partition of $Z$ into sets $V_1$ and $V_2$ using a probabilistic argument.  Partition each color class $c^{-1}(a)$ into ordered pairs arbitrarily, and denote the collection of pairs by $Q$. For each pair $(v, w) \in Q$, randomly and independently put $v$ into $V_1$ and $w$ into $V_2$ or vice versa, with probability $\frac 12$. This guarantees that $V_1$ and $V_2$ have the same coloring.

		\begin{CLAIM} \label{cl4} 
								With positive probability, the partition satisfies $\deg_{G[V_i]}(v) \geq \frac{\alpha}{160} \log d$ for all $v \in V_i$ and $i \in \{1, 2 \}$.
			\end{CLAIM} 					
					We use the Local Lemma. Fix a vertex $v$, wlog $v \in V_1$. Let $B_v$ be the event that fewer than $\frac{\alpha}{160}\log d$ neighbors of $v$ in $Z$ have ended up in $V_1$. Let $S$ be a set of $\frac{\alpha }{40}\log d$ neighbors of $v$ in $Z$, and let $T \subset S$ be the set of vertices whose match according to $Q$ does not lie in $S$. Note that $S$ contains exactly $y =\frac 12(|S|-|T|)$ pairs of $Q$. If $B_v$ occurs, then  $|S \cap V_1| \leq \frac{\alpha}{160} \log d$, and therefore $\frac 12 |S| - |S \cap V_1| \geq \frac{\alpha}{160} \log d$.
					But this implies
					$$
					\frac 12 |T| - |T \cap V_1| =
					\frac 12  \left( |S| -2y\right) - (|S \cap V_1| - y )
					\geq \frac{\alpha}{160} \log d.$$
					
					$|T \cap V_1|$ is a random variable with distribution $B\left(|T|, \frac 12\right)$, so Chernoff bounds (as stated in \cite[Remark 2.5]{JLR}) 
					give
					$$\pr{B_v} = \pr{\frac 12 |T| - |T \cap V_1| \geq \frac{\alpha}{160} \log d} \leq e^{-\frac{2}{|T|} \left( \frac{\alpha}{160} \log d \right)^2} \leq e^{-3 \log d} = d^{-3}.$$
					Here we used $|T| \leq \frac{\alpha}{40}\log d$ and $\alpha = 10^5$.
	
					Two events $B_v$ and $B_w$ are dependent only if $v$ and $w$ share a neighbor, or if some two neighbors of $v$ and $w$ are paired. In such a dependency graph, the event $B_v$ has degree at most $2d^2$. Since for sufficiently large $d$,
					$$e(2d^2 +1)d^{-3} < 1,$$
					the Lov\'asz Local Lemma grants that there is a splitting avoiding all the bad events $B_v$. This is exactly the required splitting. It concludes the proof of Claim~\ref{cl4}
					and Lemma~\ref{lemma:splitting}.
			\end{proof}
				
				We say a graph $G$ is a $p$-expander if it is connected, and
				$|N_G(U)| \geq 2|U|$ for $|U| \leq p$.
				
				\begin{LEMMA} \label{lemma:expander}
					Let $G$ be a $d$-regular graph on vertex set $[n]$ satisfying (P1) and (P2), and let $V_1 \subset [n]$ be a $G$-dense subset of $[n]$. Then $G[V_1]$ is a $\left( \frac{|V_1|}{4} \right)$-expander.
				\end{LEMMA}
				
				\begin{proof}
					Denote $H = G[V_1]$. To show expansion, suppose for the sake of contradiction that $|N_H(U)| < 2|U|$, and first assume that $|U| \leq \frac{10\log d}{d}n$. We can apply (P1) to  $T = U \cup N_H(U) \subset V_1$, using the assumption $|T| \leq \frac{30\log d}{d}n$. This gives $e(G[T]) \leq 100|T| \log d .$ Counting all the edges with an endpoint in $U$, which certainly lie inside $T$, we get
				$\frac 12 \cdot \frac{\alpha}{160}|U| \log d  \leq e(G[T]).$
				The two inequalities imply $|T| \geq \frac{1}{100}\cdot \frac{\alpha}{320}|U| = \frac{  10^5}{32000} > 3|U|$, which contradicts our assumption.
					
			        Secondly, in case $\frac{|V_1|}{4} \geq |U|\geq \frac{10\log d}{d}n$ and $|N_G(U)|< 2|U|$, we have $|V_1 \setminus (U \cup N_H(U))| \geq \frac{|V_1|}{4} \geq \frac{ 10^4\log d}{d}n $. This puts us in the position to apply  (P2) and claim that $G$ contains edges between  $U$ and $V_1 \setminus (U \cup N_H(U)) $, contradicting the definition of $N_H(U)$. Hence sets of order up to $\frac{|V_1|}{4} $ indeed expand in $H$.
			        
			       Finally, assume that $H$ is not connected. Let its smallest component be spanned by $S \subset V_1$, i.e.~$|S| \leq \frac{|V_1|}{2}$ and $N_H(S) = \emptyset$. We already showed that certainly $|S| > \frac{|V_1|}{4}$. But then the fact that $E_G(S, V_1 \setminus S) = \emptyset$ contradicts (P2).
					
				\end{proof}

				We will use these expansion properties to build long paths and ultimately a Hamilton cycle in $G$. Our approach is based on the rotation-extension technique originally developed by P\'osa.
				Given a graph $G$, denote the length (number of edges) of the longest path in $G$ by $\ell(G)$. We say that a non-edge $\{u, v\} \notin E(G)$ is a \emph{booster} with respect to $G$ if $G + \{u, v\}$ is Hamiltonian or $\ell(G + \{u, v\}) > \ell(G)$. We denote the set of boosters in $G$ by $B(G)$. 
P\'osa's rotation technique guarantees that there exist plenty of boosters in $G$
(see, for instance,  Corollary 2.10 from~\cite{lubetzky}).
				
				\begin{LEMMA}\label{lemma:boosters} Let $p$ be a positive integer. Let $G = (V, E)$ be a $p$-expander. Then  $|B(G)| \geq \frac{p^2}{2}$.
				\end{LEMMA}
				
			\subsubsection{Using $G_2$ to hit boosters in $G_1$}
				Now we move on to $G_2 = G_{n, d_2}$. Recall that we would like to add its edges to  $G_1[V_1]$ and complete a cycle on $V_1$. However, we have to argue carefully because the choice of a $G_1$-dense set $V_1$ will depend on the given vertex coloring.
				
				\begin{LEMMA} \label{lemma:hamiltonian}
					Let $G_1$ be a $d_1$-regular graph on $[n]$ with properties (P1) and (P2), 
for sufficiently large~$d_1$, and  let $\frac{d_1}{150} \leq d_2 \leq \frac{d_1}{100}$. With high probability, $G_2 = \gndt$ has the property that for any $G_1$-dense subset $V_1 \subset [n]$, $(G_1 \cup G_2)[V_1]$ is Hamiltonian.
				\end{LEMMA}
				
				The proof of Lemma~\ref{lemma:hamiltonian} consists of two parts. First we identify a deterministic property that is sufficient to make $(G_1 \cup G_2)[V_1]$ Hamiltonian, and then, using the configuration model, we show that $\gndt$ possesses this property with high probability.
				
				\begin{LEMMA} \label{lemma:ruining}
					Let $H_1$ and $H_2$ be graphs on vertex set $V_1$. Suppose that for any edge set $E^\prime \subset E(H_2)$ with $|E^\prime| \leq |V_1|$, 
					$$H_1 \cup E^\prime \text{ is Hamiltonian, \quad or \quad } B(H_1 \cup E^\prime) \cap E(H_2) \neq \emptyset.$$
					Then the graph $H_1 \cup H_2$ Hamiltonian.
				\end{LEMMA}
				
				\begin{proof}
					We will build a subset of $E(H_2)$ such that its addition to $H_1$ creates a Hamiltonian graph. Start with $E_0 = \emptyset$. Assume that $E_i$ is a subset of $i$ edges in $E(H_2)$. If the graph $H_1 \cup E_i$ is Hamiltonian, we are done. Otherwise, by hypothesis,
				$E(H_2) \cap B(H_1 \cup E_i) $ contains an edge $e$, so we set $E_{i+1} = E_i \cup \{e\}$.
				
					In each step $i$, we have $\ell(H_1 \cup E_{i+1}) > \ell(H_1 \cup E_i)$, so the process terminates after at most $|V_1|$ steps, with a Hamiltonian graph $H_1 \cup E_i$.
				\end{proof}
					

					
					\begin{LEMMA} \label{lemma:hitting}
						Let $G_1$ be a $d_1$-regular graph on $V$ with properties (P1) and (P2), where $|V|=n$ and $d_1$ is sufficiently large.
						Let $G_2 = \gndt$ for $\frac{d_1}{150} \leq d_2 \leq \frac{d_1}{100}$. We say that $G_2 \in A_{G_1}$  (or $A_{G_1}$ occurs) if there exists a $G_1$-dense subset $V_1 \subset V$, and an edge set $E \subset \binom{V_1}{2}$, $|E| \leq |V_1|$, such that $G_2$ contains $E$ and does not intersect $B((G_1+E)[V_1])$. It holds that $\pr{A_{G_1}} = o(1)$.					
					\end{LEMMA}
					
					\begin{proof}
						We will prove the claim for 
						$G_2$ sampled according to the configuration model, which is contiguous to the uniform model $\gndt$. This allows us to apply Lemma~\ref{lemma:edges}, which gives us a precise estimate on the probability of (non-)occurrence of certain edge sets. Let $P \in \pndt$ be chosen uniformly at random. We will actually bound the probability that the induced multigraph $\mathcal{M}(P)$ is in $A_{G_1}$, denoted  by $\prp{A_{G_1}}$, with a slight abuse of notation for not renaming the event $A_{G_1}$ itself.
						
						Fix a $G_1$-dense subset $V_1 \subset V$ with $|V_1| = \xi n $, and $E \subset \binom{V_1}{2}$ with $|E| = m \leq |V_1|$. Recall that since $V_1$ is $G_1$-dense, $\xi n \geq \frac{\alpha \log d_1}{4d_1} n = \frac{10^5 \log d_1}{4d_1} n$.  Note that the graph $G_1+E$ is a $\left( \frac{|V_1|}{4} \right)$-expander, so we apply Lemma~\ref{lemma:boosters}, which says that the set of boosters $B= B((G_1+E)[V_1])$ contains at least $ 2^{-5} \xi^2 n^2$ edges.
						
						Applying Lemma~\ref{lemma:edges} to $E$ and $B$, we get
						$$\prp{\mathcal{M}(P) \supset E \text{ and } \mathcal{M}(P) \cap B = \emptyset } \leq \left(\frac{2d_2}{n} \right)^m e^{-\frac {2}{5 \cdot 2^5} \xi^2  nd_2}. $$
						
						Now we can take the union bound over all choices of $E$ and $V_1$. We crudely bound the number of ways to choose $V_1$ by $n \binom{n}{\xi n}$.
						\begin{align*}
							\prp{A_{G_1}}  \leq n \binom{n}{\xi n} \sum_{m=1}^{\xi n} \binom{\frac{\xi^2 n^2}{2}}{m}   \left(\frac{2d_2}{n} \right)^m e^{-\frac{1}{5 \cdot 2^4} \xi^2  nd_2}.
						\end{align*}
						
						The term $ \binom{\frac{\xi^2 n^2}{2}}{m}   \left(\frac{2d_2}{n} \right)^m \leq  \left( \frac{e\xi ^2 n d_2}{m} \right)^m$ is increasing in $m$ in the given range, and hence
						$$\prp{A_{G_1}} \leq n \cdot \xi n \cdot \left( e\xi^{-1}  \cdot e \xi  d_2 \cdot e^{-\frac {1}{5 \cdot 2^4}\xi d_2} \right)^{\xi n } .$$
						
						Introducing the value of $\xi$, the term in brackets is at most
						$$e^2 d_2  e^{ -\frac {  10^5}{4\cdot 5 \cdot 2^4} \frac{d_2 \log d_1}{ d_1} } \leq e^2 d_2 d_1^{- \frac{300 d_2}{ d_1}}. $$
						
						For $ d_2 \in \left[\frac{d_1}{150}, \frac{d_1}{100} \right]$ the term above is upper-bounded by $ d_1^{1 -\frac{300}{150}}$, so
						$$							\prp{A_{G_1}}  \leq \xi n^2 e^{-\Omega(\xi n)} = e^{-\Omega(\xi n)},$$
						 as claimed.
					\end{proof}
				
				\begin{proof}[Proof of Lemma~\ref{lemma:hamiltonian}]
					Since $G_1$ satisfies (P1) and (P2), for $G_2 = \gndt$ it holds with high probability that $G_2 \notin A_{G_1}$. Hence, given a $G_1$-dense set $V_1 \subset V$ we can apply Lemma~\ref{lemma:ruining} to $G_1[V_1]$ and $G_2[V_1]$ to find a Hamilton cycle in $(G_1 \cup G_2)[V_1]$, as required.
				\end{proof}
				
				We are now ready to put together the proof.
				\begin{proof}[Proof of Theorem~\ref{thm:almostn}]
				    For a given $d$, set $d_2 = 2 \cdot \left \lceil \frac{d}{300} \right \rceil $ and $d_1 = d- d_2$. Let $d$ be large enough so that $d_2 \leq \frac{1}{100} d_1$, and for Lemma~\ref{lemma:quasirandom} and Lemma~\ref{lemma:hamiltonian} to hold. Moreover, by choosing $d_2$ to be even, we ensured that whenever $nd$ is even (so that $\gnd $ is non-empty), $nd_1$ and $nd_2$ are also even.
					
					Generate  $G_1 = \gndo$ and $G_2  = \gndt$ on vertex set $V$. Suppose that $G_1$ is has properties (P1) and (P2), and $G_2$ satisfies the conclusion of Lemma~\ref{lemma:hamiltonian}. By Lemma~\ref{lemma:quasirandom} and Lemma~\ref{lemma:hamiltonian}, this holds with high probability. We claim that in this case $\piper (G_1 \cup G_2) > \left(1- \frac{\alpha \log d_1}{d_1} \right)n $, where $\alpha = 10^5$ as before. Let $c: V \to \left[\left(1- \frac{\alpha\log d_1}{d_1} \right)n \right]$ be a given coloring. 
					
					We first use Lemma~\ref{lemma:splitting} to find $G_1$-dense sets $V_1, V_2 \subset V$ with the same coloring. Then by Lemma~\ref{lemma:hamiltonian}, we conclude that the graphs $(G_1 \cup G_2)[V_i]$ are Hamiltonian, for $i = 1, 2$. Let $C_1$ and $C_2$ be Hamilton cycles on $V_1$ and $V_2$. $G_1$ satisfies (P2), which implies that it contains an edge between some two vertices $v_1 \in V_1$ and $v_2 \in V_2$. We form the required path $S$ by going along $C_1$, using $v_1v_2$ to skip to $V_2$ and then going along $C_2$. The segments $S[V_1]$ and $S[V_2]$ give an anagram in $c$, as required.
					
					It remains to express the bound in terms of $d$. Note that $d_1$ lies between $\left(1 - \frac{1}{100}\right)d$ and $d$, so
					$$\frac{\alpha \log d_1}{d_1} \leq \frac{ 10^5 \log d}{d_1} \leq \frac{ 2 \cdot 10^{5} \log d}{d}. $$
					Hence $\piper (G_1 \cup G_2) > \left(1- \frac{2 \cdot 10^{5} \log d}{d} \right)n $, and by contiguity, the same holds for $\gnd$ with high probability.
				\end{proof}

        \section{Concluding Remarks}
            In this paper we studied anagram-free colorings of graphs, and showed that 
there are very sparse graphs in which anagrams cannot be avoided unless we basically 
give each vertex a separate color. 
 Our research suggests several interesting questions, some of which we mention here.
           
 The first question concerns the lower bound on the anagram-chromatic number for
trees. Is there a family of trees $T(n)$ on $n$ vertices for which  $\piper \left(T(n) \right)\ge \eps \log n$
for some positive constant $\eps>0$? We remark that this is the case  for 
the analogous problem of finding the \emph{anagram-chromatic index} of a tree.
Indeed,  a simple counting argument (cf. Proposition~\ref{prop:planar2}) shows  that
 if instead of vertex colorings, we color 
edges of a graph, then to avoid anagrams in the complete binary tree of depth $h$,
we need to use at least $\lceil \frac{1}{4} h \rceil$ colors.            
              
In estimating the anagram-chromatic number of planar graphs we relied only on 
the fact that they have small separators. It would be interesting to know 
a better lower bound on $\piper(G)$ for such graphs. In particular, we wonder if   
 there exists a family $H_n$ of planar
 graphs on $n$ vertices such that $\piper(H_n) \geq n^{\epsilon}$ for some 
 absolute constant $\epsilon >0$?
               
Let $G(n,d)$ denote the graph with  the largest 
anagram-chromatic number among all graphs  $G$ on $n$ vertices with $\Delta(G)\le d$.                 
 Our main result shows that if $d$ is large enough then 
 $\piper(G(n,d))\ge n \left( 1-C\frac{\log d}{d} \right)$, 
 while for  $d=4$ 
 we can only provide a construction which gives $\piper(G(n,4)) \geq \frac{\sqrt{n}}{2 \log n}$. 
We believe that there exist cubic graphs for which the anagram-chromatic number 
grows linearly with the order of the graph.
                
It would be nice to know how fast the function $f(d)=1-\limsup_{n\to\infty}\piper(G(n,d))/n$
decreases with $d$. Let us recall that from Proposition~\ref{prop:indep} and Theorem~\ref{thm:almostn}
it follows that 
$$\frac{1}{d}\le f(d) =  O\Big(\frac{\log d} d \Big)\, .$$ 
We do not know the correct bound, but we have good reasons to believe that the upper bound can be improved. Indeed, consider a graph 
which is a union of $2n/d$ cliques of size $d/2$ and a random $n$-vertex $d/2$-regular graph. 
We think that using such a construction one can show that
$f(d)\le (\log d)^{1/2+o(1)}/d$ but the proof looks quite involved and would probably not 
be worth the effort since it is by no means clear whether it would give the
right order of $f(d)$.
                      
Finally, Lemma~\ref{lemma:hamiltonian} motivates questions on Hamiltonicity of small induced subgraphs of $\gnd$. Pursuing our proof outline, we can prove the following.
\begin{CLAIM} \label{claim:smallsets}
    There is a constant $C$ such that with high probability, $G=\gnd$ has the following property. For any vertex set $V_1\subset [n]$ of order at least $C  \sqrt{ \frac{\log d}{d}}  n$, if the graph $H = G[V_1]$ has minimum degree at least $\frac{d}{10n}|V_1|$, then $H$ is Hamiltonian. 
\end{CLAIM}

To see this, take $G_2 = \gndt$ for $d_2 = \frac{C}{20}\sqrt{d \log d}$, and $G_1 = \gndo$ for $d_1 = d-d_2$. Consider $G = \gndo \cup \gndt$.  Let $V_1$ and $H = G[V_1]$ satisfy the hypothesis, and denote $|V_1| = \xi n$ with $\xi \geq C \sqrt{\frac{\log d}{d}}$.
Since the graph $G[V_1]$ has minimum degree at least $\frac{\xi d}{10}$, and we ensured $d_2 \leq \frac{\xi d}{20}$, $G_1[V_1]$ has minimum degree at least $\frac{\xi d}{20}$. This guarantees that $G_1[V_1]$, as well as any graph on $V_1$  containing it, has $\Theta(\xi^2 n^2)$ boosters. On the other hand,
the condition $d_2 e^{-\Omega(\xi d_2)} <1$ implies that $G_2$ hits those boosters with high probability (see the calculation in Lemma~\ref{lemma:hitting}). Hence $G[V_1]$ is Hamiltonian for any such $V_1$, and by contiguity, $\gnd$ satisfies the claim.

The above discussion leads to the natural question, what is the smallest possible lower bound on $|V_1|$ in Claim~\ref{claim:smallsets}? Note that $|V_1| = C\sqrt{\frac{\log d}{d}}n$ is the best we can get from our approach. Namely, the above-mentioned conditions require
$\Omega{\left(\frac{1}{\xi} \log \left( \frac{1}{\xi} \right) \right)} = d_2 \leq \frac{\xi d}{20}$, i.e~$\xi = \Omega \left(\sqrt{\frac{\log d}{d}}n\right)$.

We also give a lower bound on $|V_1|$. Using independent sets in $\gnd$, one can find an induced unbalanced bipartite subgraph of order $ \frac{\log d}{d} n$ with high minimum degree, which is obviously non-Hamiltonian. This observation implies that we need at least  $|V_1| \geq  \frac{\log d}{d} n$. We wonder how tight this estimate is.

\medskip

\noindent{\bf Acknowledgement.} This work was carried out when the second author visited the
Institute for Mathematical Research (FIM) of ETH Z\"urich.
He would like to thank FIM for the hospitality and for creating a stimulating research environment.
    

\end{document}